\documentclass[11pt]{amsart}

\usepackage[margin=30truemm]{geometry}
\usepackage{fancyhdr,appendix,psfrag,layout}
\usepackage{amsmath, amsthm}
\usepackage{graphicx} 
\usepackage{amssymb}
\usepackage{amstext}
\usepackage{amscd}
\usepackage{amsthm}
\usepackage{amsfonts}
\usepackage{enumerate}
\usepackage{graphicx}
\usepackage{latexsym}
\usepackage{mathrsfs}
\usepackage{physics}
\usepackage[numbers]{natbib}
\usepackage{polynom}
\usepackage{multirow}

\input xy
\xyoption{all}
\usepackage{pstricks}
\usepackage{lscape}
\usepackage{comment}
\usepackage{xcolor}
\usepackage{tikz} 
\usepackage{bbm}
\usepackage{bm}
\usepackage{tikz-cd}
\usetikzlibrary{positioning, arrows.meta, decorations.markings}
\tikzcdset{
  cells={font=\everymath\expandafter{\the\everymath\displaystyle}},
}

\usepackage[pagewise]{lineno}

\AtBeginEnvironment{subappendices}{%
\chapter*{Appendix}
\addcontentsline{toc}{chapter}{Appendices}
\counterwithin{figure}{section}
\counterwithin{table}{section}
}

\DeclareFontFamily{U}{mathc}{}
\DeclareFontShape{U}{mathc}{m}{it}%
{<->s*[1.03] mathc10}{}

\DeclareMathAlphabet{\Cal}{U}{mathc}{m}{it}

\usepackage{hyperref}
\hypersetup{
    colorlinks=true,    
    linkcolor=blue,
    filecolor=blue,
    urlcolor=orange, 
    citecolor=blue
}
\usepackage{cleveref}

\crefname{thm}{Theorem}{Theorems}
\crefname{dfn}{Definition}{Definitions}
\crefname{dfnprop}{Definition-Proposition}{Definition-Proposition}
\crefname{prop}{Proposition}{Propositions}
\crefname{lem}{Lemma}{Lemmas}
\crefname{cor}{Corollary}{Corollaries}
\crefname{clm}{Claim}{Claims}
\crefname{ass}{Assumption}{Assumption}
\crefname{cond}{Condition}{Condition}
\crefname{fct}{Fact}{Facts}
\crefname{rmk}{Remark}{Remarks}
\crefname{eg}{Example}{Examples}
\crefname{figure}{Figure}{Figures}
\crefname{table}{Table}{Tables}
\crefname{section}{Section}{Sections}
\crefname{subsection}{Subsection}{Subsections}
\crefname{appendix}{Appendix}{Appendices}
\crefname{equation}{}{}
\crefname{main}{Theorem}{Theorems}

\usepackage{amsthm}
\theoremstyle{definition}
\newtheorem{thm}{Theorem}[section]
\newtheorem{main}{Theorem}
\newtheorem{dfn}[thm]{Definition}

\newtheorem{prop}[thm]{Proposition}
\newtheorem{lem}[thm]{Lemma}
\newtheorem{cor}[thm]{Corollary}

\newtheorem{rmk}[thm]{Remark}
\newtheorem{eg}[thm]{Example}
\newtheorem*{rmk*}{Remark}

\newtheorem{cond*}{Condition}

\numberwithin{equation}{section}
\makeatletter
\let\c@equation\c@thm 
\makeatother

\newcommand{\A}{{\Cal A}}

\newcommand{\C}{{\mathsf{C}}}

\newcommand{\D}{{\mathsf{D}}}

\renewcommand{\H}{{\mathsf{H}}}

\newcommand{\K}{\mathsf{K}}
\renewcommand{\k}{\mathbf{k}}
\renewcommand{\L}{\Cal {L}}
\newcommand{\M}{{\Cal M}}

\newcommand{\T}{{\mathsf{T}}}

\newcommand{\U}{{\mathsf{U}}}

\newcommand{\X}{{\mathsf{X}}}

\newcommand{\Y}{{\Cal Y}}

\newcommand{\ZZ}{\mathbb{Z}}
\newcommand{\NN}{\mathbb{N}}

\newcommand{\Hom}{\operatorname{Hom}\nolimits}

\newcommand{\End}{\operatorname{End}\nolimits}

\renewcommand{\op}{\mathrm{op}}

\DeclareMathOperator*{\colim}{colim}

\newcommand{\ind}{\mathsf{ind}\hspace{.01in}}

\newcommand{\add}{\mathsf{add}\hspace{.01in}}

\newcommand{\Filt}{\mathsf{Filt}\hspace{.01in}}
\newcommand{\proj}{\mathsf{proj}\hspace{.01in}}

\newcommand{\length}{\mathsf{length}\hspace{.01in}}

\renewcommand{\mod}{\mathsf{mod}\hspace{.01in}}

\newcommand{\thick}{\mathsf{thick}\hspace{.01in}}

\newcommand{\Kb}{\mathsf{K}^{\rm b}}
\newcommand{\per}{\mathsf{per}\hspace{.01in}}

\newcommand{\Db}{\mathsf{D}^{\mathsf{b}}}
\newcommand{\Dfd}{\mathsf{D}_{\mathsf{fd}}}
\newcommand{\Dfl}{\mathsf{D}_{\mathsf{fl}}}

    \newcommand{\RHom}{\mathbf{R}\strut\kern-.2em\operatorname{Hom}\nolimits}
    \newcommand{\RHOM}{\mathbf{R}\strut\kern-.2em\operatorname{HOM}\nolimits}
    \newcommand{\REnd}{\mathbf{R}\strut\kern-.2em\operatorname{End}\nolimits}
    \newcommand{\REND}{\mathbf{R}\strut\kern-.2em\operatorname{END}\nolimits}

\renewcommand{\ker}{\operatorname{Ker}\nolimits}

\newcommand{\Sim}{\mathsf{Sim}}

\DeclareMathOperator*{\hocolim}{hocolim}

\newcommand{\silt}{\mathsf{silt}}

\newcommand{\SMC}{\mathsf{SMC}}

\newcommand{\xto}{\xrightarrow}

\title{The correspondence between silting objects and $t$-structures for non-positive dg algebras}
\author{Riku Fushimi}
\date{\today}

\newcommand{\Addresses}{{
  \bigskip
  \footnotesize

  R. Fushimi, \textsc{Department of mathematics, Nagoya University, Chikusa-ku, Nagoya 464-8602, Japan}\par\nopagebreak
  \textit{E-mail address}: \texttt{fushimi.riku.h9@s.mail.nagoya-u.ac.jp}

}}

\begin{document}

\begin{abstract}

We establish a bijective correspondence between isomorphism classes of basic silting objects of $\per A$ and algebraic $t$-structures of $\Dfd(A)$ for locally finite non-positive dg algebra $A$ over a field $\k$ (more generally, we work in the setting of ST-pair inside an algebraic triangulated category). For a non-positive (topologically) homologically smooth dg $\k$-algebra $A$ whose zeroth cohomology is finite-dimensional, or for a non-positive proper dg $\k$-algebra $A$, this one-to-one correspondence was already known. The main result of this paper simultaneously generalizes the above two cases.

\end{abstract}
\maketitle

\tableofcontents

\section{Introduction}

The concept of silting objects is a generalization of tilting objects. It appeared for the first time in \cite{KV88} and is often applied to the study of $t$-structure and plays important roles in representation theory. For the path algebra of Dynkin diagram $\Delta$ over a field $\k$, Keller and Vossieck \cite{KV88} established a one-to-one correspondence between isomorphism classes of basic silting objects of $\Db(\mod \k\Delta)$ and bounded $t$-structures on $\Db(\mod \k\Delta)$.  In recent years, generalizations of this bijective correspondence have been studied in various settings. For a (topologically) homologically smooth non-positive differential graded (=dg) algebra $A$ over an algebraically closed field with finite-dimensional zeroth cohomology, Keller and Nicol{\'a}s \cite{KN2} established the bijective correspondence between isomorphism classes of basic silting objects of $\per A$ and algebraic $t$-structures (=bounded $t$-structures with length heart) on $\Dfd(A)$. For a finite-dimensional $\k$-algebra $\Lambda$,  the bijective correspondence between isomorphism classes of basic silting objects of $\Kb(\proj \Lambda)$ and algebraic $t$-structures on $\Db(\mod \Lambda)$ was shown by Koenig and Yang \cite{KoY14}. For a dg $\k$-algebra $A$ with finite-dimensional total cohomology, such a bijective correspondence was shown by Su and Yang \cite{SY19} for the case $\k$ is algebraically closed via the Koszul duality between dg algebras and $A_\infty$-algebras. For the arbitrary field case, Zhang showed it in \cite{Z23} by using the dg Koszul dual and results of Keller and Nicol{\'a}s \cite{KN2}.

The notion of ST-pair was introduced in \cite{AMY19} as the natural home of silting objects and $t$-structures. S stands for silting objects and T stands for $t$-structures. For an ST-pair $(\C,\D)$, each silting object of $\C$ induces an algebraic $t$-structure on $\D$. For every locally finite non-positive dg $\k$-algebra $A$, the pair $(\per A,\Dfd(A))$ is  ST-pair (see \cref{eg:non-positive-is-ST-pair}), and the correspondence from isomorphism classes of silting objects of $\per A$ to algebraic $t$-structures on $\Dfd(A)$ is consistent with what has been studied so far. In this paper, as a main theorem, we prove that the assignment from the isomorphism classes of silting objects of $\C$ to the algebraic $t$-structures on $\D$ is bijective, and we call the bijection \emph{ST-correspondence}.

\begin{main}(\cref{thm:four-bijection})\label{thm:A}
Let $(\C,\D)$ be an ST-pair inside an algebraic triangulated category $\T$. Then there exists a bijective correspondence between
\begin{itemize}
\item[(1)] isomorphism classes of basic silting objects of $\C$,
\item[(2)] bounded co-$t$-structures on $\C$,
\item[(3)] isomorphism classes of simple-minded collections of $\D$,
\item[(4)] algebraic $t$-structures on $\D$.
\end{itemize}
\end{main}

Remark that if $\D\subseteq\C$ holds, then the above theorem can be reduced to the result of \cite{KN1} (see \cite[Corollary 6.13]{AMY19}), and if $\C\subset\D$ holds, then the above theorem can be proven in a similar manner to \cite{Z23}. The crucial novelty of this work lies in removing the compatibility between $\C$ and $\D$.

Since $(\per A,\Dfd(A))$ is an ST-pair for every non-positive dg $\k$-algebra $A$, the following theorem is a special case of \cref{thm:A}. 

\begin{main}(\cref{thm:ST-for-dg-algebra})\label{thm:B}
Let $A$ be a locally finite non-positive dg $\k$-algebra. Then there exists a bijective correspondence between
\begin{itemize}
\item[(1)] isomorphism classes of basic silting objects of $\per A$,
\item[(2)] bounded co-$t$-structures on $\per A$,
\item[(3)] isomorphism classes of simple-minded collections of $\Dfd(A)$,
\item[(4)] algebraic $t$-structures on $\Dfd(A)$.
\end{itemize}
\end{main}

\cref{thm:B} simultaneously generalizes the known results for the case where $A$ is proper or (topologically) homologically smooth with finite-dimensional zeroth-cohomology. 

\medskip
\noindent
{\bf Organization.}

In \cref{section:Preliminaries}, we collect basic definitions and facts about silting objects, simple-minded collections, dg algebras, and ST-pairs. In \cref{section:ST-correspondence-for-ST-pairs}
, we prove the main theorem.

\noindent
{\bf Conventions and notation.}

Throughout this paper, $R$ is a commutative artinian ring, and $\k$ is a field. Let us put $D(-):=\Hom_{R}(-,E)$, where $E$ is an injective hull of the direct sum of the complete set of the isomorphism classes of simple $R$-modules. Unless otherwise stated, $\T,\C$ and $\D$ are $R$-linear triangulated categories with shift functor $\Sigma$. We assume that all subcategories are full. $K_0(-)$ denotes a Grothendieck group. 

\smallskip


\section{Preliminaries}\label{section:Preliminaries}

\subsection{Silting objects and simple-minded collections}

For a subcategory $\U\subseteq\T$ (resp. an object $U\in\T$), let $\add\U$ (resp. $\add U$) denote the smallest additive subcategory of $\T$ which contains $\U$ (resp. $U$) and closed under taking direct summands, and let $\thick\U$ (resp.  $\thick U$) denote the smallest triangulated subcategory which contains $\U$ (resp. $U$) and closed under direct summands. 

\begin{dfn}\cite{AI12}
\ 
\begin{itemize}
\item[(1)] A subcategory $\M\subseteq\T$ is called a \emph{presilting subcategory} if $\Hom_{\T}(\M,\Sigma^{>0}\M)=0$ and $\M=\add\M$.
\item[(2)] An object $M\in\T$ is called a \emph{presilting object} if $\add M$ is a presilting subcategory of $\T$.
\item[(3)] A subcategory $\M\subseteq\T$ is called a \emph{silting subcategory of $\T$} if $\M$ is presilting and $\thick \M=\T$.
\item[(4)] An object $M\in\T$ is called a \emph{silting object of $\T$} if $\add M$ is a silting subcategory of $\T$.
\end{itemize}
\end{dfn}

\begin{dfn}
Let $\T$ be a Krull-Schmidt triangulated category. We define $\silt(\T)$ as the set of isomorphism classes of basic silting objects of $\T$.
\end{dfn}

For a subcategory $\X\subseteq\T$, let $\Filt(\X)$ denote the extension closure of $\X$.

\begin{dfn}\label{def:SMC}\cite{Al09}
A set $\Cal{S}=\{S_i\}_{i\in I}$ of objects of $\T$ is called a \emph{pre-simple-minded collection} if 
\begin{itemize}
\item[(1)] $\Hom_{\T}(\Cal{S},\Sigma^{<0}\Cal{S})=0$,
\item[(2)] $\Cal{S}$ is a semibrick, that is, for every $i,j\in I$,
\begin{align*}
\Hom_{\T}(S_i,S_j)=\begin{cases} \text{division ring} & \text{if } i=j,\\
0 & \text{if } i\neq j.\end{cases}
\end{align*}
\end{itemize}
We put $\H_\Cal{S}:=\Filt(\Cal{S})$. A pre-simple-minded collection $\Cal{S}$ of $\T$ is called a \emph{simple-minded collection of $\T$} if 
\begin{itemize}
\item[(3)] $\thick\Cal{S}=\T$.
\end{itemize}
\end{dfn}

\begin{dfn}
We define $\SMC(\T)$ as the set of isomorphism classes of simple-minded collections of $\T$. 
\end{dfn}

\begin{dfn}
Let $\Y_1$ and $\Y_2$ be subcategories of $\T$. Let us define $\Y_1\ast\Y_2$ as the subcategory of $\T$ consisting of objects $Z$ such that there exists an exact triangle
\begin{align*}
Y_1\to Z\to Y_2\to \Sigma Y_1,
\end{align*}
where $Y_1\in\Y_1$ and $Y_2\in\Y_2$.
\end{dfn}

\begin{dfn}
For a presilting subcategory $\M$ of $\T$, a pre-simple-minded collection $\Cal{S}$ of $\T$, a thick subcategory $\U$ of $\T$ and an integer $n$, we define full subcategories of $\T$ as follows:
\begin{eqnarray*}
\U_{\M}^{\le n}\hspace{-5pt}&:=\hspace{-5pt}&\{ X\in \U \mid \Hom_{\T}(\M,\Sigma^{>n}X)=0 \},\\
\U_{\M}^{\ge n}\hspace{-5pt}&:=\hspace{-5pt}&\{ X\in \U \mid \Hom_{\T}(\M,\Sigma^{<n}X)=0 \}, \\
\U_{\M}^{n}\hspace{-5pt}&:=\hspace{-5pt}&\U_{\M}^{\le n}\cap \U_{\M}^{\ge n}, \\
\U^{\Cal{S}}_{\le n}\hspace{-5pt}&:=\hspace{-5pt}&\{ X\in \U \mid \Hom_{\T}(\Sigma^{>n}X,\Cal{S})=0 \}, \\
\U^{\Cal{S}}_{\ge n}\hspace{-5pt}&:=\hspace{-5pt}&\{ X\in \U \mid \Hom_{\T}(\Sigma^{<n}X,\Cal{S})=0 \},\\
\U^{\Cal{S}}_{n}\hspace{-5pt}&:=\hspace{-5pt}&\U^{\Cal{S}}_{\ge n}\cap \U^{\Cal{S}}_{\le n}, 
\end{eqnarray*}
For an object $M\in\T$, we put $\U_M^{\le n}:=\U_{\add M}^{\le n}$. We define $\U_M^{\ge n}$ and $\U_M^n$ in the same way.
\end{dfn}

\begin{dfn}
For a presilting subcategory $\M$ of $\T$, a pre-simple-minded collection $\Cal{S}$ of $\T$, and an integer $n$, we define full subcategories of $\T$ as follows:
\begin{align*}
\T_{\M,\le n}&:=\T_{\M}^{\le n},\\
\T_{\M,\ge n}&:=\bigcup\nolimits_{l\ge -n}\Sigma^l\M\ast\Sigma^{l-1}\M\ast\cdots\ast\Sigma^{-n+1}\M\ast\Sigma^{-n}\M, \\
\T^{\Cal{S},\le n}&:=\T^{\Cal{S}}_{\le n}, \\
\T^{\Cal{S},\ge n}&:=\bigcup_{l\ge n}\Sigma^{-n}\H_\Cal{S}\ast\Sigma^{-n-1}\H_\Cal{S}\ast\cdots\Sigma^{-l+1}\H_\Cal{S}\ast\Sigma^{-l}\H_\Cal{S}.
\end{align*}
For an object $M\in\T$, we define $\T_{M,\le n}:=\T_{\add M,\le n}$. We define $\T_{M,\le n}$ and $\T_{M,\ge n}$ in the same way.
\end{dfn}

\begin{dfn}
Let $(\T_{\ge 0},\T_{\le 0})$ be a pair of subcategories of $\T$. Let us put $\T_{\ge n}=\Sigma^{-n}\T_{\ge 0}$, and $\T_{\le n}=\Sigma^{-n}\T_{\le 0}$. The pair $(\T_{\ge 0},\T_{\le 0})$ is called a \emph{co-t-structure} on $\T$ if 
\begin{itemize}
\item[(0)] $\T_{\ge 0}$ and $\T_{\le 0}$ are additive subcategories of $\T$ that are closed under direct summands,
\item[(1)] $\T_{\ge 1}\subseteq\T_{\ge 0}$ and $\T_{\le 0}\subseteq\T_{\le 1}$,
\item[(2)] $\Hom_{\T}(\T_{\ge 1},\T_{\le 0})=0$,
\item[(3)] $\T_{\ge 1}\ast\T_{\le 0}=\T$.
\end{itemize}
The subcategory $\T_0:=\T_{\geq0}\cap\T_{\leq0}$ is called \emph{co-heart} of the co-$t$-structure $(\T_{\geq0},\T_{\leq0})$.
\end{dfn}

\begin{dfn}
A co-t-structure $(\T_{\geq0},\T_{\leq0})$ is called \emph{bounded} if 
\begin{align*}
\bigcup_{n\in\ZZ}\T_{\geq n}=\T=\bigcup_{n\in\ZZ}\T_{\leq n},
\end{align*}
and called \emph{non-degenerate} if
\begin{align*}
\bigcap_{n\in\ZZ}\T_{\geq n}=0=\bigcap_{n\in\ZZ}\T_{\leq n}.
\end{align*}
\end{dfn}

For a given silting object, we can construct a bounded co-$t$-structure by the following proposition.
\begin{prop}\cite[Proposition 2.23]{AI12}\cite[Proposition 2.8]{IYa18}\label{cor:silting-and-co-t-structure1}
Let $\M$ be a silting subcategory of $\T$. Then we have 
\begin{align*}
\T_{\M,\le 0}=\bigcup_{l\ge 0}\M\ast\Sigma\M\ast\cdots\ast\Sigma^{l-1}\M\ast\Sigma^l\M,
\end{align*}
and $(\T_{\M,\geq0},\T_{\M,\leq0})$ is a co-$t$-structure with a co-heart $\M$.
\end{prop}

The following proposition ensures that if a Krull-Schmidt triangulated category has a silting object, then every silting subcategory has an additive generator, and the number of indecomposable direct summands of every basic silting object is constant.
\begin{prop}\cite[Proposition 2.20 and Theorem 2.27]{AI12}
\ 
\begin{itemize}
\item[(1)] If $\T$ has a silting object, then every silting subcategory has an additive generator.
\item[(2)] Let $\T$ be a Krull-Schmidt triangulated category with a silting subcategory $\M$. Then $K_0(\T)\simeq\ZZ^{\ind \M}$.
\end{itemize}
\end{prop}

\begin{prop}\cite[Corollary 5.8]{MSSS13}\label{cor:silting-and-co-t-structure2}
Let $\T$ be a Krull-Schmidt triangulated category that has a silting object. Then the map $M\mapsto(\T_{M,\geq0},\T_{M,\leq0})$ induces bijection from $\silt(\T)$ to the set of bounded co-$t$-structures on $\T$. The inverse of the above map attaches $(\T_{\geq0},\T_{\leq0})$ to the basic additive generator of $\T_0$.
\end{prop}

\begin{dfn}
Let $(\T^{\le 0},\T^{\ge 0})$ be a pair of full subcategories of $\T$. Let us put $\T^{\le n}=\Sigma^{-n}\T^{\le 0}$, and $\T^{\ge n}=\Sigma^{-n}\T^{\ge 0}$. The pair $(\T^{\le 0},\T^{\ge 0})$ is called a \emph{t-structure} on $\T$ if 
\begin{itemize}
\item[(1)] $\T^{\le 0}\subseteq\T^{\le 1}$ and $\T^{\ge 1}\subseteq\T^{\ge 0}$,
\item[(2)] $\Hom_{\T}(\T^{\le 0},\T^{\ge 0})=0$,
\item[(3)] $\T^{\le 0}\ast\T^{\ge 1}=\T$.
\end{itemize}
The subcategory $\T^0:=\T^{\le 0}\cap\T^{\ge 0}$ is called the \emph{heart} of the $t$-structure $(\T^{\le 0},\T^{\ge 0})$.
\end{dfn}

\begin{dfn}
A $t$-structure $(\T^{\le 0},\T^{\ge 0})$ is called \emph{bounded} if 
\begin{align*}
\bigcup_{n\in\ZZ}\T^{\le n}=\T=\bigcup_{n\in\ZZ}\T^{\ge n}.
\end{align*}
and called \emph{non-degenerate} if
\begin{align*}
\bigcap_{n\in\ZZ}\T^{\le n}=0=\bigcap_{n\in\ZZ}\T^{\ge n}.
\end{align*}
A bounded $t$-structure is called \emph{algebraic} if its heart is a length abelian category.
\end{dfn}

\begin{dfn}
Let $\H$ be a subcategory of $\T$.
\begin{itemize}
\item[(1)] $\H$ is called a \emph{bounded heart} if $\H$ is a heart of some bounded $t$-structure on $\T$.
\item[(2)] $\H$ is called a \emph{length heart} if $\H$ is a heart of some algebraic $t$-structure on $\T$.
\end{itemize}
\end{dfn}

The following two can be easily deduced from \cite[Remarque 1.3.14]{B82}.
\begin{lem}
Let $\Cal{S}$ be a simple-minded collection of $\T$. Then we have
\begin{align*}
\T^{\Cal{S},\le 0}=\bigcup_{l\ge 0}\Sigma^l\H_\Cal{S}\ast\Sigma^{l-1}\H_\Cal{S}\ast\cdots\ast\Sigma\H_\Cal{S}\ast\H_\Cal{S},
\end{align*}
and $(\T^{\Cal{S},\le 0},\T^{\Cal{S},\ge 0})$ is a $t$-structure on $\T$ with heart $\H_\Cal{S}$.
\end{lem}

\begin{cor}\label{cor:SMC-and-t-structure}
The map $\Cal{S}\mapsto(\T^{\Cal{S},\le 0},\T^{\Cal{S},\ge 0})$ induces a bijection from $\SMC(\T)$ to the set of algebraic $t$-structures on $\T$. The map $(\T^{\le 0},\T^{\ge 0})\mapsto\Sim(\T^0)$ gives its inverse, where $\Sim(\T^0)$ denotes the complete set of isomorphism classes of simple objects of $\T^0$.
\end{cor}

It is well known that for every bounded heart $\H$ of $\T$, the inclusion induces an isomorphism $K_0(\H)\stackrel{\simeq}{\to}K_0(\T)$ (see \cite{H88} for the case of standard $t$-structures). In particular, for every simple-minded collection $\Cal{S}$ of $\T$, we have $K_0(\T)\simeq\ZZ^{\Cal{S}}$.

\begin{rmk}\label{rmk:SMC}
Let $\Cal{S}$ and $\L$ be two simple-minded collection of $\T$. If $\L\subseteq\Filt(\Cal{S})$, then $\Cal{S}=\L$.
\end{rmk}
\begin{proof}
By assumption $\T^{\L,\le 0}\subseteq\T^{\Cal{S},\le 0}$ and $\T^{\L,\ge 0}\subseteq\T^{\Cal{S},\ge 0}$. Therefore, $(\T^{\Cal{S},\le 0},\T^{\Cal{S},\ge 0})=(\T^{\L,\le 0},\T^{\L,\ge 0})$.
\end{proof}


\subsection{Notation and facts about dg algebras}

\begin{dfn}
Let $A$ be a dg $R$-algebra.
\begin{itemize}
\item[(1)] $A$ is called \emph{locally finite} if $H^i(A)$ has finite-length for every $i\in\ZZ$.
\item[(2)] $A$ is called \emph{non-positive} if its cohomology is concentrated in non-positive part.
\item[(3)] $A$ is called \emph{positive} if its cohomology is concentrated in the non-negative part and its zeroth-cohomology is a semisimple $R$-algebra.
\end{itemize}
\end{dfn}

\begin{dfn}
Let $A$ be a dg $R$-algebra. Let $\K(A)$ denote the homotopy category of dg $A$-modules and $\D(A)$ denote the derived category of $A$ (see \cite{Ke1}). We call $\per A:=\thick A$ the \emph{perfect derived category} of $A$. We define the \emph{finite-length derived category} of $A$ by
\begin{align*}
\Dfl(A):=\{X\in\D(A)\mid\textstyle{\sum_{k\in\ZZ}}\length H^k(X)<\infty\}.
\end{align*}
If $R=k$ is a field, we put $\Dfd(A):=\Dfl(A)$ 
and call it the \emph{finite-dimensional derived category}. In addition, we put
\begin{align*}
\Dfl^-(A):=\{X\in\D(A)\mid\textstyle{\sum_{n\leq k}}\length H^k(X)<\infty,\ \text{for every } n\in\ZZ\},
\end{align*}
and we define $\Dfl^+(A)$ in a similar way.
\end{dfn}

\begin{rmk}
Let $A$ be a dg $R$-algebra.
\begin{itemize}
\item[(1)] If $A$ is non-positive, $A$ is a silting object of $\per A$.
\item[(2)] If $A$ is positive, $\ind(\add A)$ is a simple-minded collection of $\per A$, where $\ind(-)$ denotes the set of isomorphism classes of indecomposable objects.
\end{itemize}
\end{rmk}

We then define an important class of dg algebras, the (topologically) homologically smooth dg algebras.
\begin{dfn}
Let $A$ be a dg $\k$-algebra. In the below, tensor products are taken over $\k$.
\begin{itemize}
\item[(1)]\cite{KS09,G07} $A$ is called \emph{homologically smooth} if $A\in\per(A^{\op}\otimes A)$.
\item[(2)]\cite{KeY11}  $A$ is called \emph{topologically homologically smooth} if $A$ is bilaterally pseudo-compact and $A\in\per(A^{\op}\widehat{\otimes} A)$,
\end{itemize}
\end{dfn}

\begin{eg}\cite[Theorem A.17]{KeY11}
Let $(Q,W)$ be a quiver with potential. Then the complete Ginzburg dg algebra $\widehat{\Gamma}(Q,W)$ over $\k$ is topologically homologically smooth.
\end{eg}

Quivers with potential play an important role in cluster theory (see for example \cite{Am09} and \cite{BY13}).

\begin{prop}
Let $A$ be a dg $\k$-algebra. If A is homologically smooth or topologically homologically smooth, then $\Dfd(A)\subseteq\per A$ holds.
\end{prop}
\begin{proof}
For the homologically smooth case, see \cite[Lemma 4.1]{Ke2}. For the topologically homologically smooth case, see  \cite[Appendix A]{KeY11}.
\end{proof}

The following proposition plays an important role in producing silting objects from simple-minded collections. The first half part of the proposition is from \cite[Corollary 4.7]{KN1}, and the rest appeared in an earlier version of \cite{KN1} (see also \cite[Lemma 3.12]{AMY19}). 
\begin{prop}\label{prop:silting-for-positive-dg-algebra}
Let $B$ be a positive dg $R$-algebra. Then, there exists a dg $B$-module $S_B$ (unique up to isomorphism) such that the graded $H^*(B)$-module $H^*(S)$ is isomorphic to $H^0(B)$. In addition, if $B$ has a finite-length zeroth-cohomology, then $S_B$ is a silting object of $\Dfl(B)$.
\end{prop}

\begin{rmk}
In \cite{KN1} and \cite{AMY19}, the base ring $R$ is assumed to be a field, but \cref{prop:silting-for-positive-dg-algebra} and the discussion of the basic properties of ST-pairs used in the following sections can be immediately generalized to the case where $R$ is a commutative artinian ring.
\end{rmk}


\subsection{ST-pairs}
In this subsection, we recall the notion of ST-pairs and their basic properties. Most of the contents of this subsection are taken from \cite{AMY19}. 

\begin{dfn}
A triangulated $R$-category $\T$ is called \emph{Hom-finite} if $\Hom_\T(X,Y)\in\mod R$ for every $X,Y\in\T$.
\end{dfn}

\begin{dfn}\cite[Definition 4.3]{AMY19}\label{def:ST-pair}
Let $\C$ and $\D$ be thick subcategories of an idempotent complete triangulated category $\T$.
The pair $(\C,\D)$ is called an \emph{ST-pair} inside $\T$ if there exists a silting object $M$ of $\C$ such that
\begin{itemize}
\item[(ST1)]  $\Hom_{\T}(M,T)$ is finite-length for any object $T$ of $\T$, 
\item[(ST2)] $(\T_{M}^{\le 0}, \T_{M}^{\ge 0})$ is a $t$-structure on $\T$,
\item[(ST3)] $\T=\bigcup_{n\in\mathbb{Z}}\T_M^{\leq n}$ and $\D=\bigcup_{n\in\mathbb{Z}}\T_M^{\geq n}$.
\end{itemize}
When we need to emphasize the silting object $M$, we call the triple $(\C,\D, M)$ an \emph{ST-triple}. 
\end{dfn}

\begin{prop}\cite[Proposition 5.2]{AMY19}
    Let $(\C,\D,M)$ be an ST-triple inside $\T$ and let $N$ be a silting object of $\C$. Then $(\C,\D,N)$ is an ST-triple.
\end{prop}

We define $\sigma_M^{\le n}$ as the right adjoint  functor of the inclusion $\T_M^{\le n}\to\T$ and define $\sigma_M^{\ge n}$ as the left adjoint functor of the inclusion $\T_M^{\ge n}\to\T$. 

\begin{rmk}\cite[Remark 4.4(b)]{AMY19}
By (ST1), it follows that $\C$ is Hom-finite Krull-Schmidt.
\end{rmk}

We collect some examples of ST-triples. In the following, (1)-(5) are taken from \cite[Remark 5.7 and Proposition 6.12]{AMY19}, and we can check (6) by the Auslander-Reiten formula
\begin{align*}
D\Hom_{\D(A)}(X,Y)\simeq\Hom_{\D(A)}(Y,\nu X)
\end{align*}
which is natural in $X\in\per A$ and $Y\in\D(A)$, where $\nu X:=D\RHom_A(X,A)$ (see \cite[section 10]{Ke1}).
\begin{eg}\label{eg:non-positive-is-ST-pair}
Let $\Lambda$ be an artinian $R$-algebra and $A$ be a locally finite non-positive dg $R$-algebra.
\begin{itemize}
\item[(1)] If $\Lambda$ is an ordinary ring having finite global dimension, $(\Db(\mod\Lambda),\Db(\mod\Lambda),\Lambda)$ is an ST-triple inside $\Db(\mod\Lambda)$.
\item[(2)] If $A$ is topologically homologically smooth or homologically smooth, $(\per A,\Dfl(A),A)$ is an ST-triple inside $\per A$.
\item[(3)] If $\Lambda$ is an ordinary ring, $(\Kb(\proj\Lambda),\Db(\mod\Lambda),\Lambda)$ is an ST-triple inside $\Db(\mod\Lambda)$.
\item[(4)] If $A$ is proper, $(\per A,\Dfl(A),A)$ is an ST-triple inside $\Dfl(A)$.
\item[(5)] In general, $(\per A,\Dfl(A),A)$ is an ST-triple inside $\Dfl^-(A)$.
\item[(6)] In general, $((\thick DA)^{\op},\Dfl(A)^{\op},DA)$ is an ST-triple inside $\Dfl^+(A)^{\op}$.
\end{itemize}
\end{eg}

In the rest of this subsection, we fix an ST-triple $(\C,\D, M)$ inside $\T$. By the following proposition, we can get a simple-minded collection of $\D$ that is orthogonal to a given silting object $M$.

\begin{prop}\cite[Corollary 4.8]{AMY19}\label{prop:Hom-duality-between-silting-and-smc}
The heart $\T_M^0$ is Hom-finite and has a projective generator $\sigma_M^{\ge 0}(M)$. Let $\Cal{S}=\{S_1,\ldots,S_n\}:=\Sim(\T_M^0)$ and $\{M_1,\ldots,M_n\}:=\ind(\add M)$. Then $\Cal{S}$ is a simple-minded collection of $\D$, and by renumbering, we have
\begin{align*}
\Hom_\T(M_j,\Sigma^pS_i)\simeq\begin{cases} {}_{\End_{\T}(S_i)}\End_{\T}(S_i) & \text{if } i=j\ \text{and } p=0,\\
0 &  \text{otherwise}.\end{cases}
\end{align*}
\end{prop}

In the rest of this subsection, unless otherwise stated, let $\Cal{S}$ denotes the set $\{S_1,\ldots,S_n\}$, and we call it the \emph{simple-minded collection associated with $M$}. Let $S$ denote $\bigoplus_{i=1}^n S_i$.

The following proposition shows the left-right symmetry of ST-pair.

\begin{prop}\cite[Proposition 4.17]{AMY19}\label{prop:compatibility-between-M-and-T}
\ 
\begin{itemize}
\item[(1)] $\Hom_{\T}(T,S)$ is finite-length for every $T\in\T$. In particular, $\D$ is Hom-finite Krull-Schmidt.
\item[(2)] We have $(\T_{M,\ge 0},\T_{M,\leq0})=(\T^\Cal{S}_{\ge 0},\T^\Cal{S}_{\le 0})$, and it is a co-t-structure on $\T$ whose co-heart is $\add M$.
\end{itemize}
\end{prop}

Next, we consider the functorially finiteness of the heart.

\begin{dfn}
Let $\X$ be a subcategory of $\T$, and $T\in\T$. A morphism $f\colon X\to T$, with $X\in\X$, is called a \emph{right $\X$-approximation} of $T$ if 
\begin{align*}
\Hom_\T(X',f)\colon\Hom_\T(X',X)\to\Hom_\T(X',T)
\end{align*}
is surjective for every $X'\in\X$. We call $\X$ a \emph{contravariantly finite subcategory} of $\T$ if for every $T\in\T$ has a right $\X$-approximation. Dually, we difine \emph{left $\X$-approximations} and \emph{covariantly finite subcategories}. A subcategory $\X$ is called \emph{functorially finite} if $\X$ is contravariantly finite and covariantly finite.
\end{dfn}

\begin{lem}\label{lem:covariantly-finiteness-of-M-heart}
The heart $\D_M^0$ is covariantly finite in $\T$.
\end{lem}
\begin{proof}
By \cref{prop:compatibility-between-M-and-T},  $(\T_{M,\ge 0},\T_{M,\le 0})$ is a co-$t$-structure on $\T$ and hence $\T_M^{\le 0}=\T_{M,\le 0}$ is covariantly finite in $\T$. Since $(\T_M^{\le 0},\T_M^{\ge 0})$ is a $t$-structure, $\T_M^0$ is covariantly finite in $\T_M^{\le 0}$. Therefore, $\T_M^0$ is covariantly finite in $\T$.
\end{proof}

\begin{cor}\label{cor:functorially-finiteness-of-standard-heart}
Let $A$ be a locally finite non-positive dg $R$-algebra. For every silting object $M$ of $\per A$, the heart $\Dfl(A)_M^0$ is functorially finite in $\Dfl(A)$.
\end{cor}
\begin{proof}
By applying \cref{lem:covariantly-finiteness-of-M-heart} to $(\Dfl(A),\per A,M)$, it follows that $\Dfl(A)_M^0$ is covariantly finite in $\Dfl(A)$. Similarly, by applying \cref{lem:covariantly-finiteness-of-M-heart} to $(\Dfl(A)^{\op},(\thick DA)^{\op},\nu M)$, it follows that $(\Dfl(A)_M^0)^{\op}$ is covariantly finite in $\Dfl(A)^\op$.
\end{proof}


\section{ST-correspondence for ST-pairs}\label{section:ST-correspondence-for-ST-pairs}

In this section, let $(\C,\D,M)$ be an ST-triple inside an algebraic triangulated category $\T\simeq H^0(\A)$ where $\A$ is a dg enhancement of $\T$. We may assume that $\T=H^0(\A)$. Let $\Cal{S}$ be the simple-minded collection of $\D$ associated with $M$ (\cref{prop:Hom-duality-between-silting-and-smc}), and we define a positive dg $R$-algebra $B:=\Hom_\A(S,S)^{\op}$. By \cref{prop:compatibility-between-M-and-T}, the positive dg algebra $B$ is locally finite. 

\begin{dfn}
We define a contravariant triangulated functor $\Psi\colon\T\to\D(B)$ so that the following diagram is commutative:

\begin{center}
\begin{tikzpicture}[auto]
\node (11) at (0, 1.8) {$\T=H^0(\A)$}; 
\node (12) at (6.0, 1.8) {$\D(A)$};
\node (22) at (3.0, 0) {$\K(A),$};   
\draw[<-] (22) to node {$\scriptstyle \Hom_\A(-,S)$} (11);
\draw[->] (11) to node {$\scriptstyle \Psi$} (12);
\draw[->] (22) to node {$\scriptstyle $} (12);
\end{tikzpicture}
\end{center}
where $\K(A)\to\D(A)$ is the canonical quotient functor.
\end{dfn}

In \cref{subsection:positive}, we associate the ST-pair with subcategories of $\D(A)$, and in \cref{subsection:ST}, we demonstrate the ST-correspondence using the results of Keller and Nicol{\'a}s \cite{KN1}.

\subsection{ST-pairs to positive dg algebras}\label{subsection:positive}

The aim of this subsection is to prove the following proposition:

\begin{prop}\label{prop:negative-to-positive}
Let $\Psi\colon\T\to\D(B)$ be a contravariant triangulated functor defined above. 
\begin{itemize}
\item[(1)] For every $X\in \T$ and $Y\in \D$, $\Psi$ induces an isomorphism
\begin{align*}
\Hom_{\T}(X,Y)\stackrel{\simeq}{\to} \Hom_{\D(B)}(\Psi(Y),\Psi(X)),
\end{align*}
and we have $\Psi(S)\simeq B$. In particular, $\Psi$ restricts to an equivalence $\D \to \per(B)^{\op}$,
\item[(2)] For every $X\in \C$ and $Y\in \T$, $\Psi$ induces an isomorphism
\begin{align*}
\Hom_{\T}(X,Y)\stackrel{\simeq}{\to} \Hom_{\D(B)}(\Psi(Y),\Psi(X)),
\end{align*}
and we have $\Psi(M)\simeq S_B$ (see \cref{prop:silting-for-positive-dg-algebra} for the definition of $S_B$). In particular, $\Psi$ restricts to an equivalence $\C \to \Dfl(B)^{\op}$.
\end{itemize}
\end{prop}

First, we prove only (1). Then, we move on to prepare the proof of (2).

\begin{proof}[Proof of (1)]
By definition of $\Psi$, we have $\Psi(S)=B$. It suffices to show that $\Psi\colon\Hom_{\T}(X,S)\to \Hom_{\D(B)}(\Psi(S),\Psi(X))$ is an isomorphism for every $X\in\T$. By definition, $\Psi$ is decomposed as follows:
\begin{align*}
\Hom_{H^0\A}(X,S)\to\Hom_{\K(B)}(\Hom_\A(S,S),\Hom_\A(X,S))\to\Hom_{\D(B)}(\Psi(S),\Psi(X)).
\end{align*}
It is easy to check that both maps are isomorphisms.
\end{proof}

We prepare two lemmas to prove the \cref{prop:negative-to-positive} (2).

\begin{lem}\label{lem:truncation-and-limit}
Let $X\in\C$ and $Y\in\T$. Then the morphism
\begin{align*}
\Hom_{\T}(X,Y)\to\Hom_{\T}(X,\sigma_M^{\ge-n}(Y))
\end{align*}
is isomorphic for every $n\gg 0$. In particular, the canonical morphism
\begin{align*}
    \Hom_\T(X,Y)\to\lim_{n\in\NN}\Hom_\T(X,\sigma_M^{\ge-n}(Y))
\end{align*}
is an isomorphism.
\end{lem}
\begin{proof}
There exists $i\in\NN$ such that $X\in\C_{M,\ge -i}$. Then, for every $n\ge i$, we have $\Hom_{\T}(X,Y)\stackrel{\simeq}{\to}\Hom_{\T}(X,\sigma_M^{\ge -n}(Y))$.
\end{proof}

For the notion of homotopy colimit, see \cite{N01}.
\begin{lem}\label{lem:hocolim-approximation}
Let $Y\in\T$. By the property of homotopy colimit, there exists a morphism
\begin{align*}
f\colon\hocolim_{n\in\NN}\Psi(\sigma_M^{\ge -n}(Y))\to\Psi(Y)
\end{align*}
that commutes the following diagram for every $n\in\NN$:
\begin{center}
    \begin{tikzcd}
        \Psi(\sigma_M^{\geq -n}(Y)) \ar[rr] \ar[rd] & & \underset{n\in\mathbb{N}}{\hocolim}\Psi(\sigma_M^{\geq -n}(Y)) \ar[ld,"f"] \\
        & \Psi(Y). &
    \end{tikzcd}
\end{center}
Every such morphism is an isomorphism.
\end{lem}
\begin{proof}
It suffices to show that $H^i(\hocolim_{n\in\NN}\Psi(\sigma_M^{\geq -n}))\to H^i(\Psi(Y))$ is isomorphic for every $i\in\ZZ$. We have a commutative diagram:

\begin{center}
    \begin{tikzcd} 
        & \underset{n\in\mathbb{N}}{\colim} H^i(\Psi(\sigma_M^{\ge -n}(Y)))\ar[rd,"u"]\ar[dd,"v"near start] & \\
        H^i(\Psi(\sigma_M^{\ge -n}(Y))) \ar[ru]\ar[rd]\ar[rr, crossing over] &  & H^i\qty(\hocolim_{n\in\NN}\Psi(\sigma_M^{\geq -n}(Y))) \ar[ld,"H^i(f)"] \\
        & H^i(\Psi(Y)) &
    \end{tikzcd}
\end{center}
The morphism $u$ is an isomorphism because $H^i(-)$ commutes with coproducts. Therefore, it remains to show that the morphism $v$ is an isomorphism. If $n\ge i$, the canonical map $\Hom_\T(\sigma_M^{\ge -n}(Y),\Sigma^iS)\to\Hom_\T(Y,\Sigma^iS)$ is an isomorphism because $\Hom_\T(\T_M^{\le-n-1},\Sigma^iS)=0$. Since the functor $H^i(\Psi(-))=H^i(\Hom_{\A}(-,S))$ is naturally identified with $\Hom_\T(-,\Sigma^iS)$, the morphism $v$ is an isomorphism.
\end{proof}

Now we are ready to complete the proof of  \cref{prop:negative-to-positive}.
\begin{proof}[Proof of \cref{prop:negative-to-positive} (2)]
We first show that $\Psi(M)\simeq S_B$. By \cref{prop:Hom-duality-between-silting-and-smc}, we have
\begin{align*}
{}_{H^0(B)^\op}H^p(\Psi(M))\simeq{}_{\End_{H^0(\A)}(S)}\Hom_{H^0(\A)}(M,\Sigma^pS)\simeq\begin{cases} {}_{\End_{H^0(\A)}(S)}\End_{H^0(\A)}(S) & \text{if } p=0\\
0 & \text{if } p\neq0,\end{cases}
\end{align*}
and this implies $\Psi(M)\simeq S_B$ by \cref{prop:silting-for-positive-dg-algebra}.

Next, we show that
\begin{align*}
    \Psi\colon\Hom_\T(X,Y)\to\Hom_{\D(B)}(\Psi(Y),\Psi(X))
\end{align*}
is an isomorphism for every $X\in\C$ and $Y\in\T$. We have the following commutative diagram:
\begin{center}
    \begin{tikzcd}
        \Hom_\T(X,Y)\ar[r,"\Psi"]\ar[d] & \Hom_{\D(B)}(\Psi(Y),\Psi(X))\ar[d] \\
        \lim_{n\in\NN}\Hom_\T(X,\sigma_M^{\ge-n}(Y))\ar[r,"\Psi"] & \lim_{n\in\NN}\Hom_{\D(B)}(\Psi(\sigma_M^{\ge-n}(Y)),\Psi(X)).
    \end{tikzcd}
\end{center}
The left vertical map is an isomorphism by \cref{lem:truncation-and-limit}, and the bottom map is also an isomorphism by \cref{prop:negative-to-positive} (1). Therefore, it suffices to show that the right vertical map is an isomorphism. We fix a morphism
\begin{align*}
    f\colon\hocolim_{n\in\NN}\Psi(\sigma_M^{\ge -n}(Y))\to\Psi(Y)
\end{align*}
that commutes the following diagram for every $n\in\NN$:
\begin{center}
    \begin{tikzcd}
        \Psi(\sigma_M^{\geq -n}(Y)) \ar[rr] \ar[rd] & & \underset{n\in\mathbb{N}}{\hocolim}\Psi(\sigma_M^{\geq -n}(Y)) \ar[ld,"f"] \\
        & \Psi(Y). &
    \end{tikzcd}
\end{center}
By \cref{lem:hocolim-approximation}, such morphism is an isomorphism. Applying $\Hom_{\D(B)}(-,\Psi(X))$, we have the following commutative diagram:
\begin{center}
    \begin{tikzcd}
        \Hom_{\D(B)}(\Psi(Y),\Psi(X))\ar[dd]\ar[rd,"{\Hom(f,\Psi(X))}"] & \\
        & \Hom_{\D(B)}(\hocolim_{n\in\NN}\Psi(\sigma_M^{\ge-n}(Y)),\Psi(X))\ar[ld]\\ \lim_{n\in\NN}\Hom_{\D(B)}(\Psi(\sigma_M^{\ge-n}(Y)),\Psi(X)). &
    \end{tikzcd}
\end{center}
Since $\Hom(f,\Psi(X))$ is an isomorphism, the only thing left to prove is that the map
\begin{align*}
    \Hom_{\D(B)}(\hocolim_{n\in\NN}\Psi(\sigma_M^{\ge-n}(Y)),\Psi(X))\to\lim_{n\in\NN}\Hom_{\D(B)}(\Psi(\sigma_M^{\ge-n}(Y)),\Psi(X))
\end{align*}
is an isomorphism. Let 
\begin{align*}
    \alpha\colon\bigoplus_{n\in\NN}\Psi(\sigma_M^{\ge-n}(Y))\to\bigoplus_{n\in\NN}\Psi(\sigma_M^{\ge-n}(Y))
\end{align*}
be the coproduct of the canonical morphisms $\{\Psi(\sigma_M^{\ge-n}(Y)\to\Psi(\sigma_M^{\ge-(n+1)}(Y))\}_{n\in\NN}$. Applying $\Hom_{\D(B)}(-\Psi(X))$ to the exact triangle
\begin{align*}
    \bigoplus_{n\in\NN}\Psi(\sigma_M^{\ge-n}(Y))\xto{1-\alpha}\bigoplus_{n\in\NN}\Psi(\sigma_M^{\ge-n}(Y))\to\hocolim_{n\in\NN}\Psi(\sigma_M^{\ge-n}(Y))\to\Sigma\bigoplus_{n\in\NN}\Psi(\sigma_M^{\ge-n}(Y)),
\end{align*}
we have the following long exact sequence:
\begin{center}
    \begin{tikzcd}
        \prod_{n\in\NN}\Hom_{\D(B)}(\Psi(\sigma_M^{\ge-n}(Y)),\Sigma^{-1}\Psi(X))\ar[r]\ar[d, phantom, ""{coordinate, name=Z}] & \prod_{n\in\NN}\Hom_{\D(B)}(\Psi(\sigma_M^{\ge-n}(Y)),\Sigma^{-1}\Psi(X))\ar[ld,rounded corners,
        to path={ -- ([xshift=2ex]\tikztostart.east)
        |- (Z) [near end]\tikztonodes
        -| ([xshift=-2ex]\tikztotarget.west)
        -- (\tikztotarget)}] \\
        \Hom_{\D(B)}(\hocolim_{n\in\NN}\Psi(\sigma_M^{\ge-n}(Y)),\Psi(X))\ar[d, phantom, ""{coordinate, name=W}]\ar[r] & \prod_{n\in\NN}\Hom_{\D(B)}(\Psi(\sigma_M^{\ge-n}(Y)),\Psi(X))\ar[ld,rounded corners,
        to path={ -- ([xshift=2ex]\tikztostart.east)
        |- (W) [near end]\tikztonodes
        -| ([xshift=-2ex]\tikztotarget.west)
        -- (\tikztotarget)}] \\
        \prod_{n\in\NN}\Hom_{\D(B)}(\Psi(\sigma_M^{\ge-n}(Y)),\Psi(X)).
    \end{tikzcd}
\end{center}
By \cref{prop:negative-to-positive} (1) and \cref{lem:truncation-and-limit}, the morphism 
\begin{align*}
\Hom_{\D(B)}(\Psi(\sigma_M^{\ge -(n+1)}(Y)),\Psi(\Sigma X))\to\Hom_{\D(B)}(\Psi(\sigma_M^{\ge -n}(Y)),\Psi(\Sigma X)) 
\end{align*}
is isomorphic for every $n\gg 0$, and so the morphism
\begin{align*}
\prod_{n\in\NN}\Hom_{\D(B)}(\Psi(\sigma_M^{\ge -n}(Y)),\Sigma^{-1}\Psi(X))\to\prod_{n\in\NN}\Hom_{\D(B)}(\Psi(\sigma_M^{\ge -n}(Y)),\Sigma^{-1}\Psi(X))
\end{align*}
is surjective. Therefore, we have
\begin{align*}
    &\Hom_{\D(B)}(\hocolim_{n\in\NN}\Psi(\sigma_M^{\ge-n}(Y)),\Psi(X)) \\
    &\simeq
    \ker\qty(\prod_{n\in\NN}\Hom_{\D(B)}(\Psi(\sigma_M^{\ge-n}(Y)),\Psi(X))\to\prod_{n\in\NN}\Hom_{\D(B)}(\Psi(\sigma_M^{\ge-n}(Y)),\Psi(X))) \\
    &\simeq
    \lim_{n\in\NN}\Hom_{\D(B)}(\Psi(\sigma_M^{\ge-n}(Y)),\Psi(X)).
\end{align*}
\end{proof}


\subsection{ST-correspondence for ST-pairs}\label{subsection:ST}

The following is the main theorem of this subsection.

\begin{thm}\label{thm:ST-for-dg-algebra}
Let $A$ be a locally finite non-positive dg $R$-algebra. Then there exists a bijective correspondence between
\begin{itemize}
\item[(1)] isomorphism classes of basic silting objects of $\per A$,
\item[(2)] bounded co-$t$-structures on $\per A$,
\item[(3)] isomorphism classes of simple-minded collections of $\Dfl(A)$,
\item[(4)] algebraic $t$-structures on $\Dfl(A)$.
\end{itemize}
\end{thm}

To construct a silting object from a given simple-minded collection, we use the following lemma.
\begin{lem}\label{lem:smc-to-positive-dg-algebra}
For every simple-minded collection $\L$ of $\D$, there exists a locally finite positive dg $R$-algebra $B_\L$ and a contravariant triangulated functor $\Psi_\L\colon\T\to\D(B_\L)$ such that
\begin{itemize}
\item[(1)] $\Psi_\L(L)=B_\L$,
\item[(2)] $\Psi_\L$ restricts to an equivalence $\D\to\per(B_\L)^{\op}$,
\item[(3)] for every $X\in \C$ and $Y\in \D$, $\Psi_\L$ induces an isomorphism 
\begin{align*}
\Hom_{\T}(X,Y)\stackrel{\simeq}{\to} \Hom_{\D(B_\L)}(\Psi_\L(Y),\Psi_\L(X)),
\end{align*}

\item[(4)] $\Psi_\L$ restricts to an equivalence $\C \to \Dfl(B_\L)^{\op}$.
\end{itemize}
\end{lem}
\begin{proof}
By \cref{prop:negative-to-positive}, $\Psi(\L)$ is a simple-minded collection of $\per(B)$. Let $B_\L$ be a derived endomorphism ring of $\Psi(\L)$. It is easy to check that $B_\L$ is a locally finite positive dg $R$-algebra. By \cite[Lemma 6.1]{Ke1} there exists a triangulated equivalence $\Psi'_\L\colon\D(B)\to\D(B_\L)$ that sends $\Psi(\L)$ to $B_\L$. By \cref{prop:negative-to-positive}, $\Psi_\L:=\Psi'_\L\circ\Psi$ satisfies the given conditions.
\end{proof}

\begin{prop}\label{prop:silting-SMC-duality}
Let $(\C,\D)$ be an ST-pair inside an algebraic triangulated category $\T$. Then the map
\begin{align*}
\phi\colon\silt(\C)\to\SMC(\D)
\end{align*}
that sends $M$ to its associated simple-minded collection is bijective. In particular, the map $M\mapsto(\D_M^{\leq0},\D_M^{\geq0})$ induces a bijection from $\silt(\C)$ to the set of algebraic $t$-structures on $\D$.
\end{prop}
\begin{proof}
We first show that $\phi$ is injective. Let us suppose $\phi(M)=\Cal{S}=\phi(N)$. Then we have $\add M=\T^\Cal{S}_0=\add N$ by \cref{prop:compatibility-between-M-and-T}, and this implies that $M\simeq N$. Next, we show that the map $\phi$ is surjective. Let $\Cal{S}\in\SMC(\D)$. By \cref{lem:smc-to-positive-dg-algebra} and \cref{prop:silting-for-positive-dg-algebra}, there exists $M\in\silt(\C)$ such that $\Cal{S}\subseteq\T_M^0=\Filt(\phi(M))$. It follows that $\phi(M)=\Cal{S}$ by \cref{rmk:SMC}.

The rest follows from \cref{cor:SMC-and-t-structure}.
\end{proof}

\begin{rmk}
Let $M$ be a basic silting object, and $M=\bigoplus_{i=1}^nM_i$ be a decomposition into indecomposable objects. Let $\Cal{S}=\{S_1,S_2,\ldots,S_n\}$ be an associated simple-minded collection. Then we have
\begin{equation}\label{equation:orthogonal}
\Hom_\T(M_j,\Sigma^pS_i)\simeq\begin{cases} {}_{\End_{\T}(S_i)}\End_{\T}(S_i) & \text{if } i=j\ \text{and } p=0,\\
0 &  \text{otherwise}.\end{cases}
\end{equation}
In addition, $M$ and $\Cal{S}$ determine each other by the condition \cref{equation:orthogonal}.
\end{rmk}

The following corollary follows from  \cref{cor:silting-and-co-t-structure2} and \cref{prop:silting-SMC-duality}.
\begin{thm}\label{thm:four-bijection}
Let $(\C,\D)$ be an ST-pair inside an algebraic triangulated category $\T$. Then there exists a bijective correspondence between
\begin{itemize}
\item[(1)] isomorphism classes of basic silting objects of $\C$,
\item[(2)] bounded co-$t$-structures on $\C$,
\item[(3)] isomorphism classes of simple-minded collections of $\D$,
\item[(4)] algebraic $t$-structures on $\D$.
\end{itemize}
\end{thm}

Now we prove \cref{thm:ST-for-dg-algebra}
\begin{proof}[Proof of \cref{thm:ST-for-dg-algebra}]
It is a special case of \cref{thm:four-bijection}.
\end{proof}

\begin{rmk}
Let $\k$ be a field. For a locally finite non-positive dg $\k$-algebra $A$ satisfying one of the following cases, bijection between $\silt(\per A)$ and the set of algebraic $t$-structures on $\Dfd(A)$ is already known:
\begin{itemize}
\item[(1)] $A=A^0=\k\Delta$, where $\Delta$ is Dynkin quiver (see \cite{KV88}),
\item[(2)] $\k$ is algebraically closed and $\Dfd(A)\subseteq\per A$ (see \cite{KN2}),
\item[(3)] $A=A^0$ (see \cite{KoY14}),
\item[(4)] $\per A\subseteq\Dfd(A)$ (see \cite{SY19} for the case $\k$ is algebraically closed, and \cite{Z23} for the general case).
\end{itemize}
\cref{thm:ST-for-dg-algebra} generalizes all the above cases simultaneously.
\end{rmk}

The following corollary is known for the case $A$ is proper (see \cite[Proposition 3.11]{J23}).
\begin{cor}\label{cor:length-heart-is-functorially-finite}
Let $A$ be a locally finite non-positive dg $R$-algebra. Then every simple-minded collection $\Cal{S}$ of $\Dfl(A)$, and its subset $\Cal{R}$, the extension closure of $\Cal{R}$ is functorially finite in $\Dfl(A)$. 
\end{cor}
\begin{proof}
By \cref{thm:ST-for-dg-algebra}, there exists a silting object $M\in\per A$ such that $\Dfl(A)_M^0=\Filt(\Cal{S})$. By \cref{cor:functorially-finiteness-of-standard-heart}, such a heart is functorially finite in $\Dfl(A)$. It is known that $\Filt(\Cal{R})$ is functorially finite in $\Dfl(A)_M^0$ (see for example \cite[Lemma 3.12]{J23}).
\end{proof}


\section*{Acknowledgement}
The author would like to thank my supervisor Hiroyuki Nakaoka for his helpful discussions and feedback on earlier drafts. The author would like to thank Dong Yang for sharing his ideas and continuous encouragement. The author would like to thank Arashi Sakai for guiding me into such an interesting field.


\bibliographystyle{my}
\bibliography{my}

\Addresses

\end{document}